\documentclass[11pt]{amsart}

\usepackage[T1]{fontenc}
\usepackage[utf8]{inputenc}
\usepackage{lmodern}
\usepackage{amsmath,amssymb,mathtools}
\usepackage{microtype}
\usepackage{xurl}
\usepackage[hidelinks]{hyperref}

\allowdisplaybreaks

\newtheorem{theorem}{Theorem}[section]
\newtheorem{proposition}[theorem]{Proposition}
\newtheorem{lemma}[theorem]{Lemma}
\newtheorem{corollary}[theorem]{Corollary}
\theoremstyle{definition}
\newtheorem{definition}[theorem]{Definition}
\newtheorem{conjecture}[theorem]{Conjecture}
\newtheorem{openproblem}[theorem]{Open problem}
\theoremstyle{remark}
\newtheorem{remark}[theorem]{Remark}
\newtheorem{example}[theorem]{Example}

\newcommand{\F}{\mathbb F}
\newcommand{\A}{\mathbb A}
\newcommand{\Pone}{\mathbb P^1}
\newcommand{\trdeg}{\operatorname{trdeg}}
\newcommand{\Spec}{\operatorname{Spec}}
\newcommand{\Jac}{\operatorname{Jac}}
\newcommand{\Fun}{\operatorname{Fun}}
\newcommand{\LT}{\operatorname{LT}}
\newcommand{\id}{\operatorname{id}}

\title[Polynomial representatives over finite fields]
{Polynomial representatives of finite-field maps:\\
a sharp dimensional dichotomy}

\author{Stefan Bara\'nczuk}
\address{Adam Mickiewicz University, Pozna\'n, Poland}
\email{stefbar@amu.edu.pl}
\urladdr{https://orcid.org/0000-0003-4198-9241}

\author{Tomasz \'Slusarski}
\address{}
\email{tomasz.slusarski@amu.edu.pl}


\hypersetup{
  pdftitle={Polynomial representatives of finite-field maps: a sharp dimensional dichotomy},
  pdfauthor={Stefan Baranczuk and Tomasz Slusarski},
  pdfkeywords={finite fields, polynomial maps, algebraic independence, polynomial representatives, Artin--Schreier equations}
}

\begin{document}

\begin{abstract}
Let $k=\F_q$.  A polynomial representative of a finite-set map is a tuple of
polynomials inducing that map on the rational-point grid.  We prove a sharp
distinction between a finite-set map and the geometry of its representatives.
If $n=1$ or $n=2$, every polynomial representative of a permutation of $k^n$
has algebraically independent coordinates.  If $n\geq3$, every set map
$k^n\to k^n$ has both an algebraically independent and an algebraically
dependent representative; the latter may be chosen to satisfy
\[
  F_2^q-F_2=(F_1^q-F_1)F_3.
\]
More generally, every map $k^m\to k^n$ has an algebraically independent
representative exactly when $n\leq m$, while every such map has a dependent
representative when $n\geq3$.  The dependent construction combines an
Artin--Schreier interpolation theorem, producing prescribed values by
polynomials $A,B$ with $A^q-A\mid B^q-B$, with a three-coordinate suspension.
For the identity on $k^3$, the scheme-theoretic image may be chosen to be
exactly
\[
  V^q-V=(U^q-U)W,
\]
a smooth geometrically integral rational surface.  We also establish
low-degree and extension-field criteria forcing algebraic independence.  An
exact exhaustive computation additionally proves that every $2$-reduced
representative of a permutation of $\F_2^3$ has algebraically independent
coordinates.
\end{abstract}

\maketitle

\section{Introduction}\label{sec:introduction}

Let $k=\F_q$ be a finite field.  A tuple
\[
  F=(F_1,\ldots,F_n)\in k[x_1,\ldots,x_m]^n
\]
defines both a morphism
\[
  F:\A_k^m\longrightarrow\A_k^n
\]
and a map of finite sets
\[
  F:k^m\longrightarrow k^n.
\]
The morphism depends on the actual coordinate polynomials, whereas the
finite-set map depends only on their residue classes modulo the grid ideal
\[
  I_{q,m}=(x_1^q-x_1,\ldots,x_m^q-x_m).
\]
We call any polynomial tuple inducing a prescribed finite-set map a
\emph{polynomial representative}.  The distinction between the finite
function and the geometry of a chosen representative is the central issue of
this paper.

Maubach and Willems introduced mock polynomial automorphisms while studying
low-degree Keller maps over finite fields.  Their published
Conjecture~2.3 is, however, stated for an arbitrary polynomial tuple inducing
a bijection of $\F_q^n$; no Keller hypothesis appears in the conjecture
\cite[p.~480, Conjecture~2.3]{MaubachWillems2014}.  In our notation it reads
as follows.

\begin{conjecture}[Maubach--Willems statement]\label{conj:MW}
Let $k=\F_q$ and let
\[
  F=(F_1,\ldots,F_n)\in k[x_1,\ldots,x_n]^n.
\]
If the induced map $F:k^n\to k^n$ is bijective, then
$F_1,\ldots,F_n$ are algebraically independent over $k$.
\end{conjecture}

Algebraic independence of $n$ coordinates is equivalent to dominance of the
associated morphism $\A_k^n\to\A_k^n$.  Thus
Conjecture~\ref{conj:MW} asks whether bijectivity on one rational-point grid
forces global dominance.

Polynomial functions and permutation polynomials over finite fields have a
substantial classical literature; standard references include
\cite{LidlNiederreiter1997,MullenPanario2013}, and a survey of permutation
polynomials is \cite{Hou2015}.  General degree questions for interpolation
on affine and projective finite spaces have also been studied
\cite{HellusWaldi2015}.  Our problem is different: the finite values are held
fixed while the chosen representative is varied, and the question is whether
the dimension of its scheme-theoretic image can change.

Maubach and Willems proved that algebraic dependence forces the formal
Jacobian determinant to vanish \cite[Lemmas~2.4--2.5]{MaubachWillems2014};
therefore Keller maps already satisfy the conclusion of
Conjecture~\ref{conj:MW}.  Speyer gave the normalization and rational-curve
point-count argument over $\F_2$ in dimension two.  He also observed that a
higher-dimensional counterexample could arise from a polynomial map with
lower-dimensional image that remains injective on the finite grid, followed
by interpolation back to the desired finite function \cite{Speyer2014}.  We
use the same geometric obstruction uniformly over all finite fields and
realize the proposed higher-dimensional mechanism for every prescribed
finite-set map.

There is a related but distinct space-filling phenomenon.  Katz constructed
smooth strict subvarieties containing every rational point of their ambient
affine space, using Artin--Schreier equations; some auxiliary statements in
the original paper were later corrected \cite{Katz1999,Katz2001}.  Three
levels should be distinguished: a subvariety containing all ambient rational
points, a polynomial parametrization of such a subvariety, and a polynomial
representative of an arbitrarily prescribed finite function whose image is
constrained to that subvariety.  The third is the representative-theoretic
advance established here.

Our main result gives both a general source--target statement and its sharp
square-dimensional consequence.

\begin{theorem}[Polynomial representative theorem]\label{thm:main}
Let $k=\F_q$, and let $m,n\geq1$.
\begin{enumerate}
  \item Every set map $\varphi:k^m\to k^n$ has an algebraically independent
  polynomial representative if and only if $n\leq m$.
  \item If $n\geq3$, every set map $\varphi:k^m\to k^n$ has an algebraically
  dependent polynomial representative satisfying
  \begin{equation}\label{eq:main-relation}
    F_2^q-F_2=(F_1^q-F_1)F_3.
  \end{equation}
  \item In the square case $m=n$, every polynomial representative of a
  permutation has algebraically independent coordinates when $n=1,2$.
  When $n\geq3$, every set map $k^n\to k^n$ has both an algebraically
  independent and an algebraically dependent representative.
\end{enumerate}
\end{theorem}

In particular, the statement in Conjecture~\ref{conj:MW} is true in dimensions
one and two and false over every finite field in every dimension at least
three, already for the identity permutation.

The theorem is not merely the existence of one high-degree pathological
representative.  The dependent construction works for every prescribed
finite function, the same phenomenon is impossible for permutations in
dimensions one and two, and in dimension three the image can be a smooth
geometrically integral surface whose defining ideal is determined exactly.
Thus the result describes a universal and sharply dimension-dependent failure
of polynomial geometry to be intrinsic to a finite function.

The dependent construction has two steps.  First, for arbitrary prescribed
pairs
\[
  (\alpha_P,\beta_P)\in k^2,
  \qquad P\in k^m,
\]
we construct $A,B,C\in k[x_1,\ldots,x_m]$ such that
\begin{equation}\label{eq:intro-AS}
  A(P)=\alpha_P,
  \qquad
  B(P)=\beta_P,
  \qquad
  B^q-B=(A^q-A)C.
\end{equation}
The proof uses pairwise comaximal hypersurfaces, one through each rational
point, and first-order Chinese-remainder interpolation.  Second, in any
commutative $k$-algebra, put $D=A^q-A$ and define
\begin{equation}\label{eq:intro-suspension}
  F_1=A,
  \qquad
  F_2=B+D(C-T),
  \qquad
  F_3=T+D^{q-1}(C-T)^q.
\end{equation}
A direct calculation gives \eqref{eq:main-relation}.  At every rational point
where $A$ takes a value in $k$, the correction factor $D$ vanishes.  Hence the
suspension imposes a global relation without changing the prescribed
finite-set map.

For the identity on $k^3$, the construction begins in $k[x,y]$ and introduces
$z$ only in the suspension step.  We prove that the coordinate kernel is
exactly
\[
  \bigl(V^q-V-(U^q-U)W\bigr).
\]
The scheme-theoretic image is therefore a smooth geometrically integral
$k$-rational surface containing every rational point of $\A_k^3$.

We also prove two positive criteria beyond the low-dimensional obstruction.
The effective annihilator bound of Beecken, Mittmann, and Saxena
\cite[Corollary~6]{BeeckenMittmannSaxena2013} implies that a surjective square
map with coordinate degrees at most $d$ has algebraically independent
coordinates whenever $q>d^{n-1}$.  Moreover, a fixed polynomial map that is
surjective over finite extensions of unbounded cardinality has algebraically
independent coordinates.

For the canonical reduced-representative variant, we settle the smallest
three-dimensional case by an exact exhaustive computation: every
$2$-reduced representative of a permutation of $\F_2^3$ has algebraically
independent coordinates.  The computation enumerates all $8!$ permutations,
uses the formal Jacobian to settle $38{,}976$ cases, and identifies the
remaining $1{,}344$ cases as one affine postcomposition orbit of an explicitly
proved independent representative.  Exact source code and output accompany
the manuscript.

The principal new content is the divisibility-constrained Artin--Schreier
interpolation theorem, its use to construct representatives of every finite
map with one fixed algebraic relation, the resulting sharp square-dimensional
dichotomy, and the exact image realization.  We do not claim novelty for the
finite-grid algebra, the Jacobian implication, Speyer's dimension-two method,
or the elementary leading-term construction of independent representatives.
The relationship with Katz's space-filling varieties is geometric rather than
an antecedent for the universal prescribed-function theorem.

The paper is organized as follows.  Section~\ref{sec:grid} develops the
finite-grid and Jacobian preliminaries.  Section~\ref{sec:low-dimensional}
proves the low-dimensional obstruction through a general $q+1$ image bound.
Sections~\ref{sec:interpolation} and~\ref{sec:representatives} give the
interpolation, suspension, and universal representative theorems.
Section~\ref{sec:surface} identifies the exact image surface.
Section~\ref{sec:further-criteria} proves the degree and extension-field
criteria.  Section~\ref{sec:open} settles the smallest reduced-representative
case and records the remaining limitations and open problems.
Appendix~\ref{sec:compact} gives a compact characteristic-two formula.

\section{Polynomial representatives and finite-grid algebra}\label{sec:grid}

Throughout the paper,
\[
  k=\F_q,
  \qquad q=p^e,
  \qquad p=\operatorname{char}k,
\]
and for each $m\geq1$ we write
\[
  R_m=k[x_1,\ldots,x_m].
\]

\subsection{Algebraic independence and reduced representatives}

\begin{definition}
Let $\varphi:k^m\to k^n$ be a set map.  A tuple
\[
  F=(F_1,\ldots,F_n)\in R_m^n
\]
is a \emph{polynomial representative} of $\varphi$ if
$F(P)=\varphi(P)$ for every $P\in k^m$.
\end{definition}

\begin{definition}
Polynomials $f_1,\ldots,f_n\in R_m$ are \emph{algebraically independent over
$k$} if the homomorphism
\[
  k[U_1,\ldots,U_n]\longrightarrow R_m,
  \qquad U_i\longmapsto f_i,
\]
is injective.  Their transcendence degree is denoted
\[
  \trdeg_k k[f_1,\ldots,f_n].
\]
\end{definition}

\begin{lemma}[Finite-grid ideal]\label{lem:grid}
Let
\[
  I_{q,m}=(x_1^q-x_1,\ldots,x_m^q-x_m)\subset R_m.
\]
Then:
\begin{enumerate}
  \item $I_{q,m}$ is the ideal of all polynomials that vanish on $k^m$;
  \item every class in $R_m/I_{q,m}$ has a unique representative $h$
  satisfying
  \[
    \deg_{x_i}h<q
    \qquad(1\leq i\leq m);
  \]
  \item evaluation induces a $k$-algebra isomorphism
  \[
    R_m/I_{q,m}\cong\Fun(k^m,k).
  \]
\end{enumerate}
\end{lemma}

\begin{proof}
Successive Euclidean division by the monic polynomials $x_i^q-x_i$ writes
every $f\in R_m$ as
\[
  f=g+r,
  \qquad g\in I_{q,m},
  \qquad \deg_{x_i}r<q\ \text{for all }i.
\]
We prove that a reduced polynomial $r$ vanishing on $k^m$ is zero.  The proof
is by induction on $m$.  For $m=1$, a polynomial of degree less than $q$ with
$q$ roots is zero.  For $m>1$, write
\[
  r=\sum_{j=0}^{q-1}r_j(x_1,\ldots,x_{m-1})x_m^j.
\]
For each $a\in k^{m-1}$, the polynomial $r(a,x_m)$ has degree less than $q$
and vanishes at every element of $k$, hence is zero.  Thus every $r_j$ vanishes
on $k^{m-1}$, and the induction hypothesis gives $r_j=0$.

For surjectivity of evaluation, if
\[
  P=(a_1,\ldots,a_m)\in k^m,
\]
set
\begin{equation}\label{eq:indicator}
  \delta_P(x_1,\ldots,x_m)
  =\prod_{j=1}^m\bigl(1-(x_j-a_j)^{q-1}\bigr).
\end{equation}
Then $\delta_P(P)=1$ and $\delta_P(Q)=0$ for $Q\neq P$.  Linear combinations
of these point indicators represent all functions $k^m\to k$.
\end{proof}

\begin{definition}
The representative in Lemma~\ref{lem:grid}(2) is the \emph{$q$-reduced
representative} of the corresponding finite function.  Thus $q$-reduced
always means reduced in the degree of each individual source variable, not in
total degree.
\end{definition}

\subsection{The lifting formulation}

For a square tuple $F=(F_1,\ldots,F_n)\in R_n^n$, let
\[
  \Theta_F:R_n\longrightarrow R_n,
  \qquad x_i\longmapsto F_i.
\]
Since $F_i^q-F_i$ vanishes on $k^n$, the map $\Theta_F$ preserves $I_{q,n}$
and induces pullback
\[
  F^\ast:R_n/I_{q,n}\longrightarrow R_n/I_{q,n}.
\]
We also use the coordinate substitution map
\[
  \theta_F:k[U_1,\ldots,U_n]\longrightarrow R_n,
  \qquad U_i\longmapsto F_i.
\]

\begin{proposition}[Lifting reformulation]\label{prop:lifting}
For $F\in R_n^n$:
\begin{enumerate}
  \item the finite-set map $F:k^n\to k^n$ is bijective if and only if
  $F^\ast$ is an automorphism of $R_n/I_{q,n}$;
  \item the coordinates $F_1,\ldots,F_n$ are algebraically independent if
  and only if $\Theta_F$ is injective, equivalently if and only if
  $\theta_F$ is injective.
\end{enumerate}
\end{proposition}

\begin{proof}
Under Lemma~\ref{lem:grid}(3), the endomorphism $F^\ast$ is pullback of
functions along the finite-set map $F$.  Pullback on the full function algebra
is invertible exactly when $F$ is bijective.  The second assertion is the
definition of algebraic independence, after renaming the source variables.
\end{proof}

\begin{proposition}[Prime-kernel constraint]\label{prop:prime-kernel}
Suppose that $F:k^n\to k^n$ is surjective.  Then $\ker\theta_F$ is a prime
ideal contained in
\[
  (U_1^q-U_1,\ldots,U_n^q-U_n).
\]
Consequently every nonzero annihilating relation has total degree at least
$q$.
\end{proposition}

\begin{proof}
The kernel is prime because $R_n$ is a domain.  If
$H(F_1,\ldots,F_n)=0$ and $b\in k^n$, choose $a\in k^n$ with $F(a)=b$.
Then
\[
  H(b)=H(F(a))=0.
\]
Lemma~\ref{lem:grid}, applied in the $U$-variables, gives the ideal
containment.  If a polynomial has total degree less than $q$, then its degree
in every individual variable is less than $q$.  It is therefore $q$-reduced
and cannot vanish on all of $k^n$ unless it is zero.
\end{proof}

\subsection{Jacobian rank}

\begin{lemma}[Dependence forces Jacobian rank deficiency]\label{lem:jacobian-rank}
Let $k$ be a perfect field of characteristic $p>0$, and let
\[
  f_1,\ldots,f_n\in k[x_1,\ldots,x_m].
\]
If these polynomials are algebraically dependent over $k$, then their
$n\times m$ Jacobian matrix
\[
  \left(\frac{\partial f_i}{\partial x_j}\right)_{i,j}
\]
has row rank strictly less than $n$ over $k(x_1,\ldots,x_m)$.  When $n\leq m$,
a nonzero $n\times n$ Jacobian minor therefore implies algebraic independence.
\end{lemma}

\begin{proof}
Choose a nonzero relation
\[
  Q(f_1,\ldots,f_n)=0,
  \qquad Q\in k[U_1,\ldots,U_n],
\]
of minimal total degree.  If every formal partial derivative of $Q$ were zero,
then every exponent occurring in $Q$ would be divisible by $p$.  Since $k$ is
perfect, every coefficient has a $p$th root in $k$, so
\[
  Q=P^p
\]
for some polynomial $P\in k[U_1,\ldots,U_n]$ of smaller degree.  The
polynomial ring is a domain; hence $P(f_1,\ldots,f_n)=0$, contradicting
minimality.

Thus at least one $\partial Q/\partial U_i$ is nonzero.  If
\[
  \left(\frac{\partial Q}{\partial U_i}\right)(f_1,\ldots,f_n)=0,
\]
then this nonzero partial derivative would be a relation of smaller total
degree, again contradicting minimality.  Therefore
\[
  (\nabla Q)(f_1,\ldots,f_n)\neq0.
\]
Differentiating $Q(f_1,\ldots,f_n)=0$ gives
\[
  (\nabla Q)(f_1,\ldots,f_n)
  \left(\frac{\partial f_i}{\partial x_j}\right)_{i,j}=0,
\]
a nontrivial linear relation among the rows over the fraction field.
\end{proof}

\begin{remark}
For a square tuple, Lemma~\ref{lem:jacobian-rank} is the implication proved by
Maubach and Willems in \cite[Lemmas~2.4--2.5]{MaubachWillems2014}.  Their
argument treats arbitrary fields of positive characteristic by a
finite-extension descent.  The perfect-field form above suffices for finite
fields.
\end{remark}

\begin{definition}
A polynomial map $F\in k[x_1,\ldots,x_n]^n$ is a \emph{Keller map} if
\[
  \det\Jac(F)\in k^\times.
\]
It is a \emph{mock polynomial automorphism} if it is a Keller map and its
induced map $k^n\to k^n$ is bijective.
\end{definition}

\begin{corollary}[Keller maps]\label{cor:keller}
If $F=(F_1,\ldots,F_n)\in R_n^n$ satisfies
\[
  \det\Jac(F)\neq0
\]
as a polynomial, then $F_1,\ldots,F_n$ are algebraically independent.  In
particular, every Keller map and every mock polynomial automorphism satisfies
the conclusion of Conjecture~\ref{conj:MW}.
\end{corollary}

\begin{proof}
Apply the contrapositive of Lemma~\ref{lem:jacobian-rank} with $m=n$.
\end{proof}

\begin{remark}
The hypothesis in Corollary~\ref{cor:keller} is weaker than the Keller
condition: the determinant need only be a nonzero polynomial, not a nonzero
constant.
\end{remark}

\begin{example}[Zero Jacobian need not imply dependence]\label{ex:zero-jacobian}
Over $\F_2$, set
\[
  s=xy+xz+yz,
  \qquad
  (A,B,C)=(s,\ s+x+y,\ s+x+z).
\]
This map exchanges $(0,1,1)$ and $(1,0,0)$ and fixes the other six points of
$\F_2^3$.  Hence it is a permutation.  Its formal Jacobian determinant is
zero, but $A,B,C$ are algebraically independent.
\end{example}

\begin{proof}
The finite action is
\[
\begin{array}{c@{\;\longmapsto\;}c@{\qquad}c@{\;\longmapsto\;}c}
(0,0,0)&(0,0,0)&(0,0,1)&(0,0,1)\\
(0,1,0)&(0,1,0)&(0,1,1)&(1,0,0)\\
(1,0,0)&(0,1,1)&(1,0,1)&(1,0,1)\\
(1,1,0)&(1,1,0)&(1,1,1)&(1,1,1).
\end{array}
\]
For the Jacobian, add the first row to the second and third rows.  The
determinant becomes
\[
  \det
  \begin{pmatrix}
    y+z & x+z & x+y\\
    1&1&0\\
    1&0&1
  \end{pmatrix}
  =(y+z)+(x+z)+(x+y)=0.
\]
On the other hand,
\[
  x=z+A+C,
  \qquad y=z+B+C.
\]
Substitution into $A=xy+xz+yz$ gives
\[
  z^2=A+AB+AC+BC+C^2.
\]
Thus $z$, and then $x$ and $y$, are algebraic over $\F_2(A,B,C)$.  Therefore
\[
  \trdeg_{\F_2}\F_2(A,B,C)
  =\trdeg_{\F_2}\F_2(x,y,z)=3,
\]
which proves algebraic independence.
\end{proof}

\section{The low-dimensional obstruction}\label{sec:low-dimensional}

\begin{theorem}[Dimension one]\label{thm:dimension-one}
If $f\in k[x]$ induces a bijection $k\to k$, then $f$ is algebraically
independent over $k$.
\end{theorem}

\begin{proof}
The polynomial $f$ is nonconstant.  If a nonzero $H\in k[U]$ satisfied
$H(f)=0$, then $H$ would be nonconstant and
\[
  \deg H(f)=\deg H\cdot\deg f>0,
\]
a contradiction.
\end{proof}

The dimension-two argument uses standard facts about normalization and
curves.  Schemes of finite type over a field are Nagata
\cite[Tag~035B]{StacksNagata}, and the normalization of a Nagata scheme is
finite \cite[Tag~035S]{StacksNormalization}.  One-variable function fields
admit normal projective models; rational maps from normal curves to proper
varieties extend to morphisms; and over a perfect field the normal projective
model is smooth \cite[Section~53.2, Tag~0BXX]{StacksCurves}.  We also use
L\"uroth's theorem \cite{LangTate1952} and the characterization of the
projective line among genus-zero curves
\cite[Proposition~53.10.4, Tag~0C6U]{StacksProjectiveLine}.

\begin{lemma}[Dominated normal curves]\label{lem:dominated-curve}
Let $k$ be a perfect field, let $m\geq1$, and let $X$ be a normal integral
affine curve of finite type over $k$.  Suppose that there is a dominant
$k$-morphism
\[
  h:\A_k^m\longrightarrow X.
\]
Then the smooth projective model $\overline X$ of $k(X)$ is isomorphic to
$\Pone_k$.

In particular, if $k=\F_q$, then
\[
  \#X(k)\leq q+1.
\]
\end{lemma}

\begin{proof}
Dominance gives an inclusion
\[
  K:=k(X)\hookrightarrow k(x_1,\ldots,x_m).
\]
We first verify that $k$ is algebraically closed in the rational function
field on the right.  Let
\[
  r=\frac{f}{g}\in k(x_1,\ldots,x_m)
\]
be algebraic over $k$, where $f,g\in k[x_1,\ldots,x_m]$ are coprime.  There is
nothing to prove if $r=0$.  Otherwise choose a nonzero relation
\[
  c_dr^d+\cdots+c_0=0,
  \qquad c_0c_d\neq0.
\]
After multiplication by $g^d$, reduction modulo $f$ gives
\[
  f\mid c_0g^d.
\]
Since $f$ and $g$ are coprime, $f$ is a unit.  Reduction modulo $g$ similarly
shows that $g$ is a unit.  Hence $r\in k$.

It follows that $k$ is algebraically closed in $K$.  Since $k$ is perfect and
$K/k$ is finitely generated, the extension is separably generated.  Thus
$K/k$ is regular.

Let $\overline X$ be the normal projective curve with function field $K$.
It is smooth over $k$.  Regularity of $K/k$ implies that
\[
  K\otimes_k\overline{k}
\]
is a domain, so $\overline X_{\overline{k}}$ is integral.  Thus $\overline X$
is geometrically integral.  The affine curve $X$ is identified with a
nonempty open subcurve of its projective model $\overline X$; we retain these
two distinct symbols throughout the rest of the proof.

After base change to $\overline{k}$, the composite
\[
  \A_{\overline{k}}^m
  \longrightarrow X_{\overline{k}}
  \longrightarrow \overline X_{\overline{k}}
\]
is still dominant and hence nonconstant.  Working now over
$\overline{k}$, choose points $P,Q\in\A^m(\overline{k})$ with distinct images,
and let
$L\subset\A_{\overline{k}}^m$ be the affine line through them.  The
restriction
\[
  L\simeq\A_{\overline{k}}^1
  \longrightarrow \overline X_{\overline{k}}
\]
is nonconstant.  Since the target is proper, this morphism extends uniquely
to a nonconstant morphism
\[
  \Pone_{\overline{k}}
  \longrightarrow \overline X_{\overline{k}}.
\]
Consequently there is an inclusion
\[
  \overline{k}(\overline X_{\overline{k}})
  \hookrightarrow \overline{k}(t).
\]
By L\"uroth's theorem,
\[
  \overline{k}(\overline X_{\overline{k}})=\overline{k}(u)
\]
for some $u\in\overline{k}(t)$.  Hence
\[
  \overline X_{\overline{k}}\simeq\Pone_{\overline{k}},
\]
and $\overline X$ has genus zero.

The $k$-morphism $h$ supplies a $k$-rational point:
\[
  h(0)\in X(k)\subset\overline X(k).
\]
This point defines an invertible sheaf of degree one on $\overline X$.  The
characterization of the projective line cited above therefore gives
\[
  \overline X\simeq\Pone_k.
\]
If $k=\F_q$, then
\[
  \#X(k)\leq\#\overline X(k)=\#\Pone(k)=q+1.
\]
\end{proof}

\begin{proposition}[Curve-image bound]\label{prop:curve-image-bound}
Let $m,N\geq1$, and let
\[
  F=(F_1,\ldots,F_N)\in k[x_1,\ldots,x_m]^N.
\]
If
\[
  \trdeg_k k[F_1,\ldots,F_N]\leq1,
\]
then
\[
  \#F(k^m)\leq q+1.
\]
\end{proposition}

\begin{proof}
Put
\[
  R=k[x_1,\ldots,x_m],
  \qquad A=k[F_1,\ldots,F_N]\subset R.
\]
If $\trdeg_kA=0$, then every $F_i$ is algebraic over $k$.  A nonconstant
element of a polynomial ring over $k$ is transcendental over $k$, so every
$F_i$ is constant and $\#F(k^m)=1$.

Assume $\trdeg_kA=1$, put $K=\operatorname{Frac}(A)$, and let $B$ be the
integral closure of $A$ in $K$.  Since $A$ is a finitely generated $k$-algebra,
$\Spec A$ is Nagata, and the normalization is finite.  Thus
\[
  \widetilde C=\Spec B
\]
is a normal integral affine curve of finite type over $k$.

We claim that $B\subset R$.  If $b\in B$, then $b$ satisfies a monic polynomial
with coefficients in $A\subset R$, so it is integral over $R$.  Moreover,
\[
  b\in K\subset\operatorname{Frac}(R).
\]
Since $R$ is integrally closed, $b\in R$.

The inclusions $A\subset B\subset R$ give a dominant morphism
\[
  \widetilde h:\A_k^m\longrightarrow\widetilde C
\]
through which the polynomial map $F$ factors.  Hence
\[
  \#F(k^m)\leq\#\widetilde C(k).
\]
Lemma~\ref{lem:dominated-curve} gives
\[
  \#\widetilde C(k)\leq q+1.
\]
\end{proof}

\begin{theorem}[Dimension two]\label{thm:dimension-two}
Let
\[
  F=(f,g)\in k[x,y]^2.
\]
If the induced map $F:k^2\to k^2$ is injective, then $f$ and $g$ are
algebraically independent over $k$.
\end{theorem}

\begin{proof}
Injectivity gives $\#F(k^2)=q^2$.  Since $q^2>q+1$ for every $q\geq2$,
Proposition~\ref{prop:curve-image-bound} shows that
\[
  \trdeg_k k[f,g]\geq2.
\]
The opposite inequality is automatic because the algebra is generated by two
elements.  Thus its transcendence degree is two, which is equivalent to
algebraic independence of $f$ and $g$.
\end{proof}

\begin{corollary}[Transcendence-degree-one obstruction]\label{cor:trdeg-one}
Let $m\geq2$ and $N\geq1$, and let
\[
  F=(F_1,\ldots,F_N):k^m\longrightarrow k^N
\]
be polynomially represented.  If $F$ is injective, then
\[
  \trdeg_k k[F_1,\ldots,F_N]\geq2.
\]
\end{corollary}

\begin{proof}
Injectivity gives
\[
  \#F(k^m)=q^m\geq q^2>q+1.
\]
The conclusion follows from Proposition~\ref{prop:curve-image-bound}.
\end{proof}

\section{Comaximal separators and Artin--Schreier interpolation}\label{sec:interpolation}

We begin with elementary ideal arithmetic used in the central construction.

\begin{lemma}[Product of pairwise comaximal divisors]\label{lem:product-divisors}
Let $R$ be a commutative ring, and let $r_1,\ldots,r_N,f\in R$.  If the
$r_i$ are pairwise comaximal and $r_i\mid f$ for every $i$, then
\[
  r_1\cdots r_N\mid f.
\]
\end{lemma}

\begin{proof}
For two factors, write $f=ra=sb$ and choose $u,v\in R$ with $ur+vs=1$.  Then
\[
  f=f(ur+vs)=rs(bu+av).
\]
For each $i<N$, choose $a_i,b_i\in R$ with $a_ir_N+b_ir_i=1$.  Multiplying
these identities shows that $r_N$ is comaximal with $r_1\cdots r_{N-1}$.  The
two-factor argument and induction prove the result.
\end{proof}

\begin{lemma}[Coprimality with a product]\label{lem:coprime-product}
Let $R$ be a commutative ring.  If $(E,r_i)=R$ for every $i$, then
\[
  \left(E,\prod_{i=1}^N r_i\right)=R.
\]
\end{lemma}

\begin{proof}
Choose $u_i,v_i\in R$ with $Eu_i+r_iv_i=1$.  Multiplying these identities,
every term except $\prod_i r_iv_i$ contains a factor $E$.  Thus
\[
  1=EU+\left(\prod_i r_i\right)\left(\prod_i v_i\right)
\]
for some $U\in R$.
\end{proof}

\begin{lemma}[Combining coprime divisors]\label{lem:combine-divisors}
Let $R$ be a commutative ring.  If
\[
  (E,\Delta)=R,
  \qquad E\mid f,
  \qquad \Delta\mid f,
\]
then $E\Delta\mid f$.
\end{lemma}

\begin{proof}
Write $f=Ea=\Delta b$ and choose $u,v$ with $Eu+\Delta v=1$.  Then
\[
  f=f(Eu+\Delta v)=E\Delta(bu+av).
\]
\end{proof}

Let $m\geq1$, set $R=k[x_1,\ldots,x_m]$, and enumerate
\[
  k^m=\{P_1,\ldots,P_N\},
  \qquad N=q^m.
\]
Write $P_i=(a_{i1},\ldots,a_{im})$ and define
\[
  \delta_i=\prod_{j=1}^m\bigl(1-(x_j-a_{ij})^{q-1}\bigr).
\]
Then $\delta_i(P_j)=1$ when $i=j$ and $0$ otherwise.  Define recursively
\begin{equation}\label{eq:rho-recursion}
  \Pi_0=1,
  \qquad \rho_i=1-\Pi_{i-1}\delta_i,
  \qquad \Pi_i=\Pi_{i-1}\rho_i.
\end{equation}

\begin{lemma}[Separating hypersurfaces]\label{lem:separators}
The polynomials $\rho_1,\ldots,\rho_N$ satisfy
\[
  \rho_i(P_i)=0,
  \qquad \rho_i(P_j)=1\quad(j\neq i),
\]
and are pairwise comaximal.  Consequently their squares are pairwise
comaximal.
\end{lemma}

\begin{proof}
At $P_i$, every preceding $\rho_j$ has value $1$, so
$\Pi_{i-1}(P_i)=1$ and $\rho_i(P_i)=0$.  If $j\neq i$, then
$\delta_i(P_j)=0$, so $\rho_i(P_j)=1$.

If $j<i$, then $\rho_j\mid\Pi_{i-1}$ and
\[
  1=\rho_i+\Pi_{i-1}\delta_i\in(\rho_i,\rho_j).
\]
Thus the $\rho_i$ are pairwise comaximal.  If
$u\rho_i+v\rho_j=1$, expanding $(u\rho_i+v\rho_j)^3$ expresses $1$ as a sum
of a multiple of $\rho_i^2$ and a multiple of $\rho_j^2$; hence the squares
are pairwise comaximal.
\end{proof}

\begin{theorem}[Artin--Schreier interpolation]\label{thm:AS-interpolation}
For arbitrary prescribed values
\[
  \alpha_i,\beta_i\in k
  \qquad(1\leq i\leq N),
\]
there exist $A,B,C\in R$ such that
\[
  A(P_i)=\alpha_i,
  \qquad B(P_i)=\beta_i
\]
for every $i$, and
\begin{equation}\label{eq:AS-divisibility}
  B^q-B=(A^q-A)C.
\end{equation}
This includes the cases in which either prescribed function is constant.
\end{theorem}

\begin{proof}
For each $i$, put
\[
  M_i=\prod_{j\neq i}\rho_j^2.
\]
Since the $\rho_j^2$ are pairwise comaximal,
Lemma~\ref{lem:coprime-product} gives $(M_i,\rho_i^2)=R$.  Choose
$\lambda_i\in R$ with
\[
  \lambda_iM_i\equiv1\pmod{\rho_i^2}.
\]
Define
\begin{equation}\label{eq:A-construction}
  A=\sum_{i=1}^N(\alpha_i+\rho_i)\lambda_iM_i.
\end{equation}
For each $i$,
\begin{equation}\label{eq:A-local}
  A\equiv\alpha_i+\rho_i\pmod{\rho_i^2},
\end{equation}
because every other summand is divisible by $\rho_i^2$.  Hence
$A(P_i)=\alpha_i$.

The construction must keep $A$ nonconstant even when the prescribed first
coordinate is constant.  If the $\alpha_i$ are not all equal, this is
immediate.  If all $\alpha_i$ equal $\alpha$ and $A=\alpha$, then
\eqref{eq:A-local} implies $\rho_i\in(\rho_i^2)$.  Since $R$ is a domain and
$\rho_i\neq0$, this would make $\rho_i$ a unit, contradicting
$\rho_i(P_i)=0$.

Set $D=A^q-A$.  Since $A$ is nonconstant,
\[
  \deg(A^q)=q\deg A>\deg A,
\]
so $D\neq0$.  From \eqref{eq:A-local}, write
$A=\alpha_i+\rho_i+\rho_i^2h_i$.  Using $\alpha_i^q=\alpha_i$ and $q\geq2$
gives
\[
  D\equiv-\rho_i\pmod{\rho_i^2}.
\]
Therefore
\begin{equation}\label{eq:D-local}
  \rho_i\mid D,
  \qquad \frac{D}{\rho_i}\equiv-1\pmod{\rho_i}.
\end{equation}
By Lemma~\ref{lem:product-divisors},
\[
  \Delta:=\prod_{i=1}^N\rho_i
\]
divides $D$.  Write $D=\Delta E$.  Equation~\eqref{eq:D-local} yields
\[
  \left(\frac{\Delta}{\rho_i}\right)E\equiv-1\pmod{\rho_i}.
\]
Because $\Delta/\rho_i$ is a unit modulo $\rho_i$, this congruence shows that
$E$ is a unit modulo $\rho_i$.  Thus $(E,\rho_i)=R$ for every $i$, and
Lemma~\ref{lem:coprime-product} gives
\begin{equation}\label{eq:E-Delta-coprime}
  (E,\Delta)=R.
\end{equation}

Define
\begin{equation}\label{eq:B-construction}
  B=-\sum_{i=1}^N\beta_i\frac{D}{\rho_i}.
\end{equation}
Modulo $\rho_j$, every term with $i\neq j$ vanishes because $D/\rho_i$
contains the factor $\rho_j$.  The remaining term and
\eqref{eq:D-local} give
\[
  B\equiv\beta_j\pmod{\rho_j}.
\]
Thus $B(P_j)=\beta_j$.

Since $B\equiv\beta_i\pmod{\rho_i}$ and $\beta_i^q=\beta_i$, every $\rho_i$
divides $B^q-B$.  Lemma~\ref{lem:product-divisors} gives
\[
  \Delta\mid B^q-B.
\]
Every summand in \eqref{eq:B-construction} is divisible by $E$, so $E\mid B$.
Writing $B=EB_0$ gives
\[
  B^q-B=E\bigl(E^{q-1}B_0^q-B_0\bigr),
\]
so $E\mid B^q-B$.  Finally, \eqref{eq:E-Delta-coprime} and
Lemma~\ref{lem:combine-divisors} imply
\[
  D=E\Delta\mid B^q-B.
\]
Therefore
\[
  C=\frac{B^q-B}{D}\in R,
\]
and \eqref{eq:AS-divisibility} holds.
\end{proof}

\begin{remark}
First-order interpolation controls the simple local factors of $A^q-A$; the
second interpolation then makes $B^q-B$ divisible by all factors of
$A^q-A$.  The recursion \eqref{eq:rho-recursion} is designed for uniform
existence, not degree efficiency, and produces large coordinates even over
small fields.
\end{remark}

\section{Universal polynomial representatives}\label{sec:representatives}

\subsection{Artin--Schreier suspension}

\begin{lemma}[Artin--Schreier suspension]\label{lem:suspension}
Let $R$ be a commutative $k$-algebra, and let $A,B,C,T\in R$ satisfy
\[
  B^q-B=(A^q-A)C.
\]
Set $D=A^q-A$ and define
\begin{equation}\label{eq:suspension}
  F_1=A,
  \qquad F_2=B+D(C-T),
  \qquad F_3=T+D^{q-1}(C-T)^q.
\end{equation}
Then
\begin{equation}\label{eq:suspension-relation}
  F_2^q-F_2=(F_1^q-F_1)F_3.
\end{equation}
\end{lemma}

\begin{proof}
The $q$th-power map is a ring homomorphism.  Hence
\begin{align*}
  F_2^q-F_2
   &=B^q-B+D^q(C-T)^q-D(C-T)\\
   &=DC+D^q(C-T)^q-DC+DT\\
   &=D\bigl(T+D^{q-1}(C-T)^q\bigr)\\
   &=(F_1^q-F_1)F_3.
\end{align*}
\end{proof}

\subsection{Dependent representatives}

\begin{theorem}[Universal constrained representatives]\label{thm:dependent-representatives}
Let $m\geq1$ and $n\geq3$, and let
\[
  \varphi=(\varphi_1,\ldots,\varphi_n):k^m\longrightarrow k^n
\]
be any set map.  There is a polynomial representative
\[
  F=(F_1,\ldots,F_n)\in R_m^n
\]
of $\varphi$ satisfying
\begin{equation}\label{eq:universal-relation}
  F_2^q-F_2=(F_1^q-F_1)F_3.
\end{equation}
Consequently its coordinates are algebraically dependent.
\end{theorem}

\begin{proof}
Enumerate $k^m=\{P_1,\ldots,P_N\}$.  Apply
Theorem~\ref{thm:AS-interpolation} in $R_m$ to the prescribed values
\[
  \alpha_i=\varphi_1(P_i),
  \qquad \beta_i=\varphi_2(P_i).
\]
This gives $A,B,C\in R_m$ with
\[
  A(P_i)=\alpha_i,
  \qquad B(P_i)=\beta_i,
  \qquad B^q-B=(A^q-A)C.
\]
For $3\leq j\leq n$, choose a polynomial representative $T_j$ of
$\varphi_j$, for example
\[
  T_j=\sum_{i=1}^N\varphi_j(P_i)\delta_{P_i}.
\]
Put $D=A^q-A$ and define
\[
  F_1=A,
  \qquad F_2=B+D(C-T_3),
  \qquad F_3=T_3+D^{q-1}(C-T_3)^q,
\]
and $F_j=T_j$ for $j\geq4$.  Lemma~\ref{lem:suspension} gives
\eqref{eq:universal-relation}.

At every $P_i$,
\[
  D(P_i)=\alpha_i^q-\alpha_i=0,
\]
so $F(P_i)=\varphi(P_i)$.  The nonzero polynomial
\[
  U_2^q-U_2-(U_1^q-U_1)U_3
  \in k[U_1,\ldots,U_n]
\]
annihilates the coordinates of $F$.
\end{proof}

\begin{corollary}[Counterexamples over every finite field]\label{cor:all-fields}
For every finite field $k$ and every $n\geq3$, the identity map of $k^n$ has
a polynomial representative with algebraically dependent coordinates
satisfying \eqref{eq:universal-relation}.  Thus the statement in
Conjecture~\ref{conj:MW} is false in every dimension at least three.
\end{corollary}

\begin{proof}
Apply Theorem~\ref{thm:dependent-representatives} with $m=n$ and
$\varphi=\id_{k^n}$.
\end{proof}

\subsection{Independent representatives}

\begin{theorem}[Universal independent representatives]\label{thm:independent-representatives}
Let $m,n\geq1$, and let $\varphi:k^m\to k^n$ be a set map.  Then $\varphi$
has an algebraically independent polynomial representative if and only if
$n\leq m$.
\end{theorem}

\begin{proof}
If $n>m$, no $n$ elements of $R_m$ can be algebraically independent because
$\trdeg_k R_m=m$.

Assume $n\leq m$.  Choose any polynomial representative
\[
  f=(f_1,\ldots,f_n)
\]
of $\varphi$.  Let $d\geq0$ bound the total degrees of the nonzero $f_i$,
with $d=0$ if all coordinates vanish.  Choose $N\geq1$ such that $qN>d$ and
define
\[
  G_i=f_i+(x_i^q-x_i)^N,
  \qquad 1\leq i\leq n.
\]
The correction terms vanish on $k^m$, so $G=(G_1,\ldots,G_n)$ represents
$\varphi$.

Fix a monomial order refining total degree.  Since $qN>d$,
\[
  \LT(G_i)=x_i^{qN}.
\]
For $a=(a_1,\ldots,a_n)\in\mathbb N^n$,
\[
  \LT(G_1^{a_1}\cdots G_n^{a_n})
  =x_1^{qNa_1}\cdots x_n^{qNa_n}.
\]
Distinct exponent vectors give distinct leading monomials.  If
\[
  H(U_1,\ldots,U_n)=\sum_a c_aU^a
\]
is nonzero, choose an exponent $a$ for which $c_a\neq0$ and the displayed
leading monomial of $G^a$ is maximal.  It occurs exactly once in
$H(G_1,\ldots,G_n)$ and cannot cancel.  Hence $H(G_1,\ldots,G_n)\neq0$, so
the coordinates are algebraically independent.
\end{proof}

\subsection{The sharp square-dimensional dichotomy}

\begin{corollary}[Sharp representative dichotomy]\label{cor:dichotomy}
Let $\varphi:k^n\to k^n$ be a set map.
\begin{enumerate}
  \item If $n\geq3$, $\varphi$ has both an algebraically independent and an
  algebraically dependent polynomial representative.
  \item If $n=1$ or $n=2$ and $\varphi$ is bijective, every polynomial
  representative of $\varphi$ has algebraically independent coordinates.
\end{enumerate}
Consequently Theorem~\ref{thm:main} holds.
\end{corollary}

\begin{proof}
The first assertion follows from Theorems~\ref{thm:dependent-representatives}
and~\ref{thm:independent-representatives}.  The second follows from
Theorems~\ref{thm:dimension-one} and~\ref{thm:dimension-two}.
\end{proof}

\begin{corollary}[Ring-theoretic form]\label{cor:ring-form}
Let $A_{q,n}=R_n/I_{q,n}$.  Every $k$-algebra automorphism of $A_{q,n}$ has
an injective endomorphism lift to $R_n$.  If $n\geq3$, it also has a
noninjective endomorphism lift.  If $n\leq2$, every endomorphism lift of an
automorphism of $A_{q,n}$ is injective.
\end{corollary}

\begin{proof}
Under
\[
  A_{q,n}\cong\Fun(k^n,k)\cong k^{q^n},
\]
a $k$-algebra automorphism permutes the primitive idempotents and is therefore
pullback along a unique permutation of $k^n$.  With the convention that
pullback by $\sigma$ is $f\mapsto f\circ\sigma$, this may be the inverse of
the permutation induced on primitive idempotents; the contravariance does not
affect the existence statements.  Apply Proposition~\ref{prop:lifting} and
Corollary~\ref{cor:dichotomy}.
\end{proof}

\begin{definition}\label{def:scheme-image}
Let $F=(F_1,\ldots,F_n):\A_k^m\to\A_k^n$ be a polynomial map, and let
\[
  \theta_F:k[U_1,\ldots,U_n]\longrightarrow R_m,
  \qquad U_i\longmapsto F_i.
\]
The \emph{scheme-theoretic image} of $F$ is the closed subscheme of $\A_k^n$
defined by $\ker\theta_F$.  Its coordinate ring is
\[
  k[U_1,\ldots,U_n]/\ker\theta_F
  \cong k[F_1,\ldots,F_n].
\]
\end{definition}

\section{An identity representative with exact image}\label{sec:surface}

Apply Theorem~\ref{thm:AS-interpolation} in $k[x,y]$ with prescribed values
\[
  A(a,b)=a,
  \qquad B(a,b)=b
  \qquad((a,b)\in k^2).
\]
Write
\[
  D=A^q-A,
  \qquad B^q-B=DC.
\]
In $k[x,y,z]$, define
\begin{equation}\label{eq:identity-representative}
  F_1=A,
  \qquad F_2=B+D(C-z),
  \qquad F_3=z+D^{q-1}(C-z)^q.
\end{equation}

\begin{theorem}[Exact scheme-theoretic image]\label{thm:exact-surface}
The map $F=(F_1,F_2,F_3)$ in
\eqref{eq:identity-representative} has the following properties.
\begin{enumerate}
  \item $F(a,b,c)=(a,b,c)$ for every $(a,b,c)\in k^3$.
  \item The coordinate kernel is exactly
  \[
    (\Psi_q),
    \qquad
    \Psi_q(U,V,W)=V^q-V-(U^q-U)W.
  \]
  \item The scheme-theoretic image is the hypersurface
  \[
    S_q:\quad V^q-V=(U^q-U)W.
  \]
  It is smooth, geometrically integral, and $k$-rational.
  \item As subsets of the ambient rational-point set,
  \[
    S_q(k)=\A_k^3(k).
  \]
\end{enumerate}
\end{theorem}

\begin{proof}
At a rational point $(a,b,c)$,
\[
  A(a,b)=a,
  \qquad B(a,b)=b,
  \qquad D(a,b)=a^q-a=0,
\]
so $F(a,b,c)=(a,b,c)$.  Lemma~\ref{lem:suspension} gives
\begin{equation}\label{eq:surface-relation}
  F_2^q-F_2=(F_1^q-F_1)F_3.
\end{equation}
Thus $\Psi_q$ belongs to the coordinate kernel.

We first prove that $F_1,F_2$ are algebraically independent.  The polynomial
$A$ is nonconstant because it assumes every first-coordinate value on $k^2$;
hence $D\neq0$.  Moreover,
\[
  F_2=B+DC-Dz
\]
is linear in $z$ with nonzero leading coefficient $-D$.  Suppose
\[
  H(F_1,F_2)=0,
  \qquad H(U,V)=\sum_{j=0}^s h_j(U)V^j,
  \qquad h_s\neq0.
\]
If $s>0$, the coefficient of $z^s$ in $H(A,F_2)$ is
\[
  h_s(A)(-D)^s,
\]
which is nonzero because a nonconstant polynomial $A$ is transcendental over
$k$.  If $s=0$, then $h_0(A)\neq0$ for the same reason.  Thus $F_1,F_2$ are
algebraically independent.

Since $F_1^q-F_1=D\neq0$, relation~\eqref{eq:surface-relation} gives
\[
  F_3=\frac{F_2^q-F_2}{F_1^q-F_1}\in k(F_1,F_2).
\]
Therefore the image coordinate algebra has transcendence degree at most two;
independence of $F_1,F_2$ gives transcendence degree exactly two.  Since the
coordinate kernel is prime, the dimension formula gives
\[
  \operatorname{ht}\ker\theta_F
  =3-\trdeg_k k[F_1,F_2,F_3]=1.
\]

We next prove geometric irreducibility of $\Psi_q$.  Let $\ell/k$ be any field
extension and regard $\Psi_q$ as a polynomial in $W$ over $\ell[U,V]$.  The
coefficients $-(U^q-U)$ and $V^q-V$ are coprime.  Indeed, every nonconstant
irreducible divisor of the first coefficient lies in $\ell[U]$, while every
nonconstant irreducible divisor of the second lies in $\ell[V]$; a common
irreducible divisor would therefore belong to
\[
  \ell[U]\cap\ell[V]=\ell,
\]
which is impossible.  Thus $\Psi_q$ is primitive in $\ell[U,V][W]$.  Over
$\ell(U,V)$ it is a nonconstant polynomial of degree one in $W$, and hence
irreducible.  Gauss's lemma gives irreducibility in $\ell[U,V,W]$.  Therefore
$(\Psi_q)$ is a height-one prime contained in the coordinate kernel.  The
containment cannot be strict: together with the zero prime, a strict
containment would give a chain of two proper prime inclusions and force the
coordinate kernel to have height at least two.  Hence the two height-one
primes are equal.

Since
\[
  \frac{\partial\Psi_q}{\partial V}=-1,
\]
the hypersurface is smooth.  On the dense open set $U^q-U\neq0$,
\[
  W=\frac{V^q-V}{U^q-U},
\]
so its function field is $k(U,V)$ and the surface is $k$-rational.  Finally,
for every $(u,v,w)\in k^3$,
\[
  u^q-u=v^q-v=0,
\]
so every ambient rational point lies on $S_q$.
\end{proof}

\begin{remark}
Theorem~\ref{thm:exact-surface} exhibits a smooth geometrically integral
strict closed affine surface containing every rational point of its ambient
affine space.  The dimension-two proof succeeds not by naive dimension
counting, but because a normalized image curve dominated by affine space is
rational and has at most $q+1$ rational points.
\end{remark}

\section{Further criteria forcing algebraic independence}\label{sec:further-criteria}

\subsection{Low degree over a large field}

\begin{theorem}[Low-degree/large-field criterion]\label{thm:degree-criterion}
Let
\[
  F=(F_1,\ldots,F_n)\in k[x_1,\ldots,x_n]^n
\]
be surjective on $k^n$.  Assume that
\[
  \deg F_i\leq d
  \qquad(1\leq i\leq n),
\]
where degree means total degree and $d\geq1$.  If
\[
  q>d^{n-1},
\]
then $F_1,\ldots,F_n$ are algebraically independent.
\end{theorem}

\begin{proof}
Assume that the coordinates are dependent and put
\[
  r=\trdeg_k k[F_1,\ldots,F_n]\leq n-1.
\]
The effective annihilator bound of Beecken, Mittmann, and Saxena applies
over an arbitrary field: for a list of polynomials of maximum total degree
$d$ and transcendence degree $r$, algebraic dependence yields a nonzero
annihilator of total degree at most $d^r$
\cite[Corollary~6]{BeeckenMittmannSaxena2013}.  Hence there is a nonzero
polynomial
\[
  H\in k[U_1,\ldots,U_n]
\]
such that
\[
  H(F_1,\ldots,F_n)=0,
  \qquad \deg H\leq d^r\leq d^{n-1}<q.
\]
Surjectivity implies that $H$ vanishes at every point of $k^n$.  Every
individual variable degree of $H$ is less than $q$, so
Lemma~\ref{lem:grid} forces $H=0$, a contradiction.
\end{proof}

\begin{corollary}[Necessary degree of a dependent permutation representative]\label{cor:degree-lower}
Let $n\geq2$.  If a polynomial representative of a permutation of $k^n$ has
algebraically dependent coordinates of maximum total degree $d$, then
\[
  d^{n-1}\geq q
  \qquad\text{and hence}\qquad
  d\geq\left\lceil q^{1/(n-1)}\right\rceil.
\]
In dimension three,
\[
  d\geq\lceil\sqrt q\rceil.
\]
\end{corollary}

\begin{proof}
This is the contrapositive of Theorem~\ref{thm:degree-criterion}.
\end{proof}

\begin{remark}
The strict inequality in Theorem~\ref{thm:degree-criterion} is essential to
that argument.  The equality case $q=d^{n-1}$ is not decided by the theorem.
\end{remark}

\subsection{Surjectivity over unbounded extensions}

\begin{theorem}[Surjectivity over unbounded extensions]\label{thm:extensions}
Let
\[
  F\in\F_q[x_1,\ldots,x_n]^n.
\]
Suppose that $F$ is surjective on $K^n$ for finite extensions $K/\F_q$ of
unbounded cardinality.  The extensions need not form a nested tower.  Then
the coordinates of $F$ are algebraically independent over $\F_q$.
\end{theorem}

\begin{proof}
Suppose that a nonzero
\[
  H\in\F_q[U_1,\ldots,U_n]
\]
satisfies $H(F)=0$.  Surjectivity on $K^n$ implies that $H$ vanishes on all of
$K^n$.  Choose one of the given extensions with
\[
  |K|>\max_i\deg_{U_i}H.
\]
The finite-grid argument of Lemma~\ref{lem:grid}, now over $K$, gives $H=0$,
a contradiction.
\end{proof}

\section{Reduced representatives and open problems}\label{sec:open}

Theorem~\ref{thm:dependent-representatives} shows that algebraic independence
is not a property of a finite-set map when arbitrary representatives are
allowed.  A natural possible repair is to select the canonical reduced
representative.  We first settle the smallest three-dimensional case.

\subsection{The reduced \texorpdfstring{$\F_2^3$}{F2 cubed} case}

\begin{proposition}[Computer-assisted reduced case]\label{prop:reduced-f2}
Let $F=(F_1,F_2,F_3)$ be the unique $2$-reduced representative of a
permutation of $\F_2^3$.  Then $F_1,F_2,F_3$ are algebraically independent
over $\F_2$.
\end{proposition}

\begin{proof}
By Lemma~\ref{lem:grid}, each coordinate function on $\F_2^3$ has a unique
multilinear algebraic normal form.  The accompanying exact-check script
enumerates all $8!=40{,}320$ permutations and computes the formal Jacobian
determinant of each reduced polynomial triple in $\F_2[x,y,z]$, without
imposing the grid relations $x^2=x$, $y^2=y$, or $z^2=z$.  It gives
\[
\begin{array}{c|r}
\text{reduced permutation representatives} & 40{,}320\\
\text{nonzero formal Jacobian determinant} & 38{,}976\\
\text{zero formal Jacobian determinant} & 1{,}344.
\end{array}
\]
The first $38{,}976$ triples are algebraically independent by
Corollary~\ref{cor:keller}.

The same exact enumeration verifies that the remaining $1{,}344$ triples are
precisely the orbit, under invertible affine postcomposition, of
\[
  G=(A,B,C),
  \qquad
  A=xy+xz+yz,
  \quad B=y+z,
  \quad C=x+z+1.
\]
The script compares the two sets directly.  Affine linear combinations of the
multilinear coordinates remain $2$-reduced, so equality of the induced
permutation sets is equality of the reduced polynomial triples, not merely
equality of cardinalities.  In $\F_2[x,y,z]$ one has
\[
  x=C+z+1,
  \qquad y=B+z,
  \qquad z^2+A+BC+B=0.
\]
On $\F_2^3$ the last identity recovers $z=A+BC+B$, and the first two then
recover $x$ and $y$, so $G$ induces a permutation.  Globally, the same
identities show that $\F_2(x,y,z)$ is algebraic over $\F_2(A,B,C)$.  Hence
\[
  \trdeg_{\F_2}\F_2(A,B,C)=3,
\]
so $A,B,C$ are algebraically independent.  Invertible affine
postcomposition preserves algebraic independence.  It therefore handles
every zero-Jacobian case and completes the exhaustive proof.
\end{proof}

\begin{remark}[Reproducibility]\label{rem:reduced-reproducibility}
The supplementary file
\path{Paper_A_f2_reduced_check.py} uses only the Python standard library and
exact arithmetic over $\F_2$.  It converts truth tables to
algebraic normal form by the Boolean M\"obius transform, represents formal
polynomials as sparse exponent sets, differentiates and multiplies them
without Boolean reduction, enumerates every permutation exactly once, and
compares the zero-Jacobian set with all $168\cdot8=1{,}344$ invertible affine
postcompositions of $G$.  The exact console output is supplied as
\path{Paper_A_f2_reduced_check_output.txt}.  Checksums and the execution
command are recorded in \path{Paper_A_supplement_README.md} and
\path{Paper_A_reproducibility_manifest.sha256}.  No computation is used
in the proofs of Theorems~\ref{thm:main},~\ref{thm:AS-interpolation},
or~\ref{thm:exact-surface}.
\end{remark}

\subsection{Remaining problems}

\begin{openproblem}[Reduced-representative variant]\label{problem:reduced}
Let $F=(F_1,\ldots,F_n)$ be the unique $q$-reduced representative of a
permutation of $\F_q^n$.  Must $F_1,\ldots,F_n$ be algebraically independent?
\end{openproblem}

The answer is affirmative for $n=1,2$, because
Theorems~\ref{thm:dimension-one} and~\ref{thm:dimension-two} apply to every
representative, and Proposition~\ref{prop:reduced-f2} gives an affirmative
answer for $(q,n)=(2,3)$.  Apart from this case, the present paper leaves the
question open for $n\geq3$.  Ring-theoretically, one must construct or rule
out a $q$-reduced permutation representative whose coordinate kernel is a
positive-height prime contained in
\[
  (U_1^q-U_1,\ldots,U_n^q-U_n).
\]
Reduction modulo the grid ideal can destroy the exact relation satisfied by a
nonreduced representative, so the universal construction does not settle the
question.

\begin{openproblem}[Minimum degree]\label{problem:min-degree}
For $n\geq3$, determine
\[
  d_{q,n}=\min_F\max_i\deg F_i,
\]
where $F$ ranges over polynomial representatives of permutations of $\F_q^n$
with algebraically dependent coordinates.
\end{openproblem}

Corollary~\ref{cor:degree-lower} gives
\[
  d_{q,n}\geq\left\lceil q^{1/(n-1)}\right\rceil.
\]
Proposition~\ref{prop:compact-f2} below gives $d_{2,3}\leq14$, but no minimum
is determined here.

\begin{openproblem}[Compact formulas for $q>2$]
Find substantially lower-degree closed-form representatives over $\F_q$,
$q>2$, preferably with exact image surface
\[
  V^q-V=(U^q-U)W.
\]
The general separator construction proves existence but is not
degree-efficient.
\end{openproblem}

\begin{openproblem}[Prime relation kernels]
Classify the nonzero prime ideals
\[
  P\subset(U_1^q-U_1,\ldots,U_n^q-U_n)
\]
that occur as coordinate kernels of polynomial representatives of
permutations of $\F_q^n$.
\end{openproblem}

\begin{openproblem}[Descent]
Let $q=p^e$ with $e>1$.  Determine when a dependent representative of a
permutation of $\F_q^n$ can be chosen with coefficients in a proper subfield,
especially in $\F_p$.
\end{openproblem}

\section{Conclusion}\label{sec:conclusion}

A polynomial tuple over a finite field carries two different kinds of data:
its values on the finite rational-point grid and the global geometry of the
associated morphism.  Theorem~\ref{thm:main} gives a sharp square-dimensional
dichotomy.  In dimensions one and two, every representative of a permutation
is dominant.  In every dimension at least three, each finite-set map has both
a dominant and a nondominant representative.

The nondominant representatives arise from a reusable construction rather
than an isolated example.  Artin--Schreier interpolation produces prescribed
values with $A^q-A\mid B^q-B$, and the suspension converts that divisibility
into the fixed relation
\[
  F_2^q-F_2=(F_1^q-F_1)F_3.
\]
For the identity on $k^3$, the relation defines the exact scheme-theoretic
image, a smooth geometrically integral rational surface containing every
ambient $k$-point.

Several positive restrictions remain: nonzero formal Jacobian determinant,
low coordinate degree relative to $q$, and surjectivity over extensions of
unbounded size each force algebraic independence.  In the canonical reduced
setting, Proposition~\ref{prop:reduced-f2} gives an exact computer-assisted
proof for permutations of $\F_2^3$.  The general reduced-representative
question remains Open Problem~\ref{problem:reduced}.

\appendix

\section{A compact characteristic-two representative}\label{sec:compact}

The uniform construction has large degree.  We now give a smaller explicit
formula over $\F_2$.  In $\F_2[x,y,z]$, define
\begin{align*}
  r&=xy+1,
  &s&=xy+y+1,
  &\theta&=x^2+x+1,\\
  a&=x+1+s\theta,
  &b&=x+r\theta,
  &\mu&=rs,\\
  \nu&=ab,
  &u&=\theta^2+\nu z,
  &v&=1+\mu z,
\end{align*}
and set
\begin{equation}\label{eq:compact-map}
  F_1=ra,
  \qquad F_2=\mu u+1,
  \qquad F_3=uv+1.
\end{equation}

\begin{proposition}[Compact characteristic-two representative]\label{prop:compact-f2}
The map $F=(F_1,F_2,F_3)$ in \eqref{eq:compact-map} has the following
properties.
\begin{enumerate}
  \item On $\F_2^3$ it induces the affine transvection
  \[
    (x,y,z)\longmapsto(x+y,y,z).
  \]
  \item Its coordinates satisfy
  \[
    F_2(F_2+1)=F_1(F_1+1)(F_3+1).
  \]
  \item Its coordinate degrees are
  \[
    \deg F_1=6,
    \qquad \deg F_2=13,
    \qquad \deg F_3=14.
  \]
  \item Its coordinate kernel is generated by
  \[
    \Phi(X,Y,Z)=Y(Y+1)+X(X+1)(Z+1).
  \]
  Consequently its scheme-theoretic image is a smooth geometrically integral
  rational surface.
  \item After every field extension of $\F_2$, the formal Jacobian has rank
  exactly two at every geometric point.
\end{enumerate}
\end{proposition}

\begin{proof}
In characteristic two,
\[
  r(x+1)+sx=1.
\]
Therefore
\[
  ra+sb=r(x+1)+sx+rs\theta+sr\theta=1.
\]
Also
\begin{align*}
  \nu=ab
   &=(x+1+s\theta)(x+r\theta)\\
   &=x(x+1)+\theta\bigl(r(x+1)+sx\bigr)+rs\theta^2\\
   &=1+\mu\theta^2.
\end{align*}
Hence
\[
  \mu\theta^2+\nu=1
\]
and
\[
  \mu u+\nu v
  =\mu(\theta^2+\nu z)+\nu(1+\mu z)=1.
\]
It follows that
\[
  F_1+1=sb,
  \qquad F_2+1=\mu u,
  \qquad F_2=\nu v,
  \qquad F_3+1=uv.
\]
Multiplication gives
\[
  F_2(F_2+1)
  =(\nu v)(\mu u)
  =(ab)(rs)uv
  =(ra)(sb)uv
  =F_1(F_1+1)(F_3+1).
\]

Reduce modulo $(x^2+x,y^2+y,z^2+z)$.  Then
\[
  \theta=1,
  \qquad a=xy+x+y,
  \qquad \mu=y+1,
  \qquad \nu=y,
\]
and
\[
  u=1+yz,
  \qquad v=1+(y+1)z.
\]
Substitution gives
\[
  F_1=x+y,
  \qquad F_2=y,
  \qquad F_3=z
\]
as functions on $\F_2^3$.

The highest-degree terms of
\[
  r,s,\theta,a,b,\mu,\nu,u,v
\]
are respectively
\[
  xy,\quad xy,\quad x^2,\quad x^3y,\quad x^3y,\quad x^2y^2,
  \quad x^6y^2,\quad x^6y^2z,\quad x^2y^2z.
\]
Thus the highest-degree terms of $F_1,F_2,F_3$ are
\[
  x^4y^2,
  \qquad x^8y^4z,
  \qquad x^8y^4z^2,
\]
which proves the degree assertions.

We compute a constant Jacobian minor.  Since
\[
  F_1=r(x+1)+rs\theta,
\]
we have
\begin{align*}
  \frac{\partial F_1}{\partial y}
   &=x(x+1)+\bigl(xs+r(x+1)\bigr)\theta\\
   &=x(x+1)+\theta=1,
\end{align*}
because $xs+r(x+1)=1$.  Also
\[
  \frac{\partial F_1}{\partial z}=0,
  \qquad
  \frac{\partial F_3}{\partial z}=\nu v+\mu u=1.
\]
Hence the $(F_1,F_3)$ minor in the variables $(y,z)$ equals $1$.
Lemma~\ref{lem:jacobian-rank} implies that $F_1,F_3$ are algebraically
independent.

The displayed relation shows that the coordinate algebra has transcendence
degree at most two, while independence of $F_1,F_3$ gives transcendence degree
exactly two.  If
\[
  \vartheta:\F_2[X,Y,Z]\longrightarrow\F_2[x,y,z]
\]
is the coordinate substitution map, then
\[
  \operatorname{ht}\ker\vartheta
  =3-\trdeg_{\F_2}\F_2[F_1,F_2,F_3]=1.
\]
Thus the coordinate kernel is a height-one prime containing $\Phi$.

Let $L/\F_2$ be any field extension and view $\Phi$ as a polynomial in $Z$
over $L[X,Y]$.  Its coefficients
\[
  X(X+1)
  \quad\text{and}\quad
  Y(Y+1)+X(X+1)
\]
are coprime, because a common factor would divide both $X(X+1)$ and
$Y(Y+1)$.  Hence $\Phi$ is primitive in $L[X,Y][Z]$.  It has degree one in
$Z$ over $L(X,Y)$ and is irreducible there; Gauss's lemma gives irreducibility
in $L[X,Y,Z]$.  Thus $\Phi$ is geometrically irreducible and $(\Phi)$ is a
height-one prime.  Since it is contained in the height-one coordinate kernel,
a strict containment would force the latter to have height at least two.
Therefore the coordinate kernel equals $(\Phi)$.

Smoothness follows from
\[
  \frac{\partial\Phi}{\partial Y}=1,
\]
and rationality follows by solving for $Z$ on $X(X+1)\neq0$.  Differentiating
the relation gives a polynomial dependence among the three Jacobian rows
whose coefficient on the $F_2$ row is $1$, so after every scalar extension
the Jacobian rank is at most two at every geometric point.  The constant minor
remains equal to $1$ and proves that the rank is at least two everywhere.
\end{proof}

\begin{proposition}[Failure of literal scalar transport]\label{prop:compact-extension}
Let $k=\F_{2^e}$ with $e>1$, and interpret the formula
\eqref{eq:compact-map} over $k$.  The resulting map is not bijective on
$k^3$.
\end{proposition}

\begin{proof}
The polynomial relation remains valid, so the image lies in
\[
  S:\quad Y(Y+1)=X(X+1)(Z+1).
\]
If $X\in\{0,1\}$, then $Y\in\{0,1\}$ and $Z$ is arbitrary, giving $4q$
points.  If $X\notin\{0,1\}$, then for every $Y$ there is a unique $Z$, giving
$q(q-2)$ points.  Therefore
\[
  \#S(k)=4q+q(q-2)=q^2+2q<q^3
\]
for $q>2$.  The map cannot be surjective.
\end{proof}

\begin{remark}
The compact formula relies on the cancellation $2=0$.  The uniform
construction in Sections~\ref{sec:interpolation}--\ref{sec:surface} does not:
for each finite field it is rebuilt using $T^q-T$.  It is not generally
obtained by scalar extension of a representative over a smaller field.
\end{remark}

\section*{Declaration of generative AI and AI-assisted technologies}
During preparation of this work, the authors used ChatGPT for exploratory
mathematical brainstorming, manuscript organization and language editing,
LaTeX assistance, literature-search support, and computational checking.
The exact exhaustive computation used in
Proposition~\ref{prop:reduced-f2} is documented by the supplementary source,
output, execution instructions, and checksums identified in
Remark~\ref{rem:reduced-reproducibility}.  The authors reviewed and
independently verified all material retained in the article and take full
responsibility for its content.

\bibliographystyle{amsplain}
\bibliography{Paper_A_final}

\end{document}